\documentclass[11pt]{article}

\usepackage{amsfonts,amsmath,amsthm,amssymb}
\usepackage{pstricks,amsmath,here,latexsym,amssymb}

\usepackage{epsfig}
\usepackage{enumerate}
\usepackage{a4wide}
\usepackage{makeidx}
\usepackage{color}
\usepackage{psfrag}

\newcommand{\mC}{\mathcal{C}}
\newcommand{\mF}{\mathcal{F}}
\newcommand{\mG}{\mathcal{G}}

\newcommand{\N}{\mathbb{N}}

\newcommand{\Cat}{\textrm{Cat}}
\newcommand{\Ap}{\textrm{Ap}}
\newcommand{\Ex}{\textrm{Ex}}

\newcommand{\ite}{\noindent $\bullet~$}
\newcommand{\iten}{\noindent -~}

\newcommand{\eqexp}{\asymp_{\textrm{exp}}}

\newtheorem{thm}{Theorem}
\newtheorem{prop}[thm]{Proposition}
\newtheorem{claim}[thm]{Claim}
\newtheorem{lem}[thm]{Lemma}

\title{On the growth rate of minor-closed classes of graphs}
\author{Olivier Bernardi\footnote{CNRS, D\'{e}partement de Math\'{e}matiques
Universit\'{e} Paris-Sud 91405 Orsay Cedex, France. {\tt
olivier.bernardi@gmail.com} }, \and Marc Noy\footnote{Universitat
Polit\`{e}cnica de Catalunya, Jordi Girona 1--3, 08034 Barcelona,
Spain. {\tt marc.noy@upc.edu}}
 \and Dominic Welsh\footnote{University of Oxford, Mathematical Institute,
 24-29 St Giles', Oxford OX1 3LB, UK. {\tt dwelsh@maths.ox.ac.uk }}
}

\date{\today}

\begin{document}

\maketitle

\begin{abstract}
A minor-closed  class of graphs is a set of labelled graphs which
is closed under isomorphism and under taking minors. For a
minor-closed class $\mG$, we let $g_n$ be the number of graphs in
$\mG$ which have $n$ vertices. A recent result of Norine \emph{et
al.} \cite{Norine:small} shows that for all minor-closed class
$\mG$, there is a constant $c$ such that $g_n\leq c^n n!$. Our
main results show that the growth rate of $g_n$ is far from
arbitrary. For example, no minor-closed class $\mG$ has $g_n=
c^{n+o(n)} n!$ with $0<c<1$ or $1<c<\xi\approx 1.76$.
\end{abstract}

\section{Introduction}
In 1994,  Scheinerman and Zito \cite{Scheinerman:speed-hereditary}
introduced the study of  the possible growth rates of hereditary
classes of graphs (that is, sets of graphs which are closed under
isomorphism and induced subgraphs). Here we study the same problem
for classes which are closed under taking minors. Clearly, being
minor-closed is a much stronger property than to be  hereditary.
However, many  of  the  more  structured  hereditary  classes such
as  graphs embeddable  in a  fixed surface  or graphs of  tree
width bounded by a  fixed constant are minor-closed  and the
possible growth rates attainable are of independent  interest.

A broad  classification  of  possible growth rates for hereditary
classes given by  Scheinermann and Zito
\cite{Scheinerman:speed-hereditary} is into four  categories,
namely  constant, polynomial,  exponential and  factorial. This
has been considerably  extended  in a series of  papers by Balogh,
Bollobas  and  Weinrich
\cite{Bolagh:speed-hereditary,Bolagh:hereditary-penultimate,Bolagh:hereditary-Bell}
who  use  the  term \emph{speed} for  what we  call growth  rate.

A first and important  point to  note  is  that if a class of
graphs is  minor-closed then it is hereditary. Hence, in what
follows  we are working within the  confines described by the
existing classifications of  growth rates of hereditary  classes.
Working in this more restricted context, we obtain simpler
characterization of the different categories of growth rate and
simpler proofs.  This is done in Section
\ref{section:classification-thm}. In Section
\ref{section:growth-constants}, we establish some results about
the possible behaviour about classes in the most interesting range
of  growth rates,  namely  the  factorial range. We conclude by
listing some open questions in Section \ref{section:conclusion}.

A  significant  difference between hereditary and minor-closed classes is  due  to the following
recent result by Norine \emph{et al.} 
A class is proper if it does not contain all graphs.

\begin{thm}[Norine et al. \cite{Norine:small}]\label{th:small}
If $\mathcal{G}$ is a proper minor-closed class of graphs then $g_n \le c^n n! $ for some constant $c$.
\end{thm}

\paragraph{Remark.} In contrast, a hereditary class such as the set of bipartite
graphs can have growth rate of order $2^{cn^2}$ with $c>0$.

We close this introduction with some definitions and notations. We
consider simple labelled graphs. The \emph{size} of a graph is the
number of vertices; graphs of size $n$ are labelled with vertex
set $\{1,2,\dots,n\}$. A \emph{class} of graphs is a family of
labelled graphs closed under isomorphism. For a class of graphs
$\mathcal{G}$, we let $\mathcal{G}_n$ be the graphs in
$\mathcal{G}$ with $n$ vertices, and we let $g_n
=|\mathcal{G}_n|$. The (exponential) \emph{generating function}
associated to a class $\mG$ is $G(z)=\sum_{n\geq  0}
\frac{g_n}{n!} z^n$.

The relation $H < G$ between graphs means \emph{$H$ is a minor of
$G$}. A family $\mathcal{G}$ is \emph{minor-closed} if $G \in
\mathcal{G}$ and $H<G$ implies $H \in \mathcal{G}$.  A class is
\emph{proper} if it does not contain all graphs. A graph $H$ is a
(minimal) \emph{excluded minor} for a minor-closed family
$\mathcal{G}$ if $H \not\in \mathcal{G}$ but every proper minor of
$H$ is in $\mathcal{G}$. We write $\mathcal{G} =\Ex(H_1,H_2,
\cdots)$ if $H_1,H_2,\dots$ are the excluded minors of
$\mathcal{G}$. By the theory of graph minors developed by
Robertson and Seymour \cite{Seymour:Graph-minors}, the number of
excluded minors is always finite.

\section{A classification theorem}\label{section:classification-thm}

Our classification theorem for the possible growth rate of minor-closed classes of graphs involves the following classes;
it is easy to check that they are all minor-closed.\\
\ite $\mathcal{P}$ is the class of \emph{path forests}: graphs whose connected components are paths.\\
\ite $ \mathcal{S}$ is the class of \emph{star forests}: graphs whose connected components are stars (this includes isolated vertices).\\
\ite $ \mathcal{M}$ is the class of \emph{matchings}: graphs whose connected components are edges and isolated vertices.\\
\ite $ \mathcal{X}$ is the class of \emph{stars}: graphs made of one star and some isolated vertices.

\begin{thm}\label{th:refine}
Let $\mG$ be a proper minor-closed family and let $g_n$ be the number of graphs in $\mathcal{G}$ with $n$ vertices.
\begin{enumerate}
\item If $\mG$ contains all the paths, then $g_n$ has \emph{factorial growth}, that is, \\
  $n!   \leq g_n  \leq  c^n n^! \textrm{ for some } c>1;$    \label{item:factorial}
\item else, if $\mG$ contains all the star forests, then $g_n$ has \emph{almost-factorial growth}, that is,\\
$B(n)  \leq  g_n \leq \epsilon^n n!~  \textrm{ for all } \epsilon>0$,  where $B(n)$ is the  $n^{\rm th}$ Bell number;
\item else, if $\mG$ contains all the matchings, then $g_n$ has \emph{semi-factorial growth}, that is, \\
$a^n n^{(1-1/k)n}   \leq g_n  \leq  b^n n^{(1-1/k)n}~\textrm{ for some integer } k\geq 2 \textrm{ and some } a,b>0;$
\item else, if $\mG$ contains all the stars, then   $g_n$ has \emph{exponential growth}, that is,\\
$2^{n-1} \leq g_n \leq c^n~ \textrm{ for some } c>2;$          \label{item:exponential}
\item else, if $\mG$ contains all the graphs with a single edge, then   $g_n$ has \emph{polynomial growth}, that is,
 $g_n = P(n)~\textrm{ for some polynomial } P(n)  \textrm{ of degree at least 2  and } n \textrm{ sufficiently large};$
\item else, $g_n$ is \emph{constant}, namely
$g_n \textrm{ is equal to  0 or 1 for } n \textrm{ sufficiently large}.$
\end{enumerate}
\end{thm}

\paragraph{Remark.} As mentioned in the introduction, some of the results given by Theorem \ref{th:refine}
follow from the  previous  work  on  hereditary  classes. In
particular, the classification of growth between \emph{pseudo
factorial} (this includes our categories factorial,
almost-factorial and semi-factorial),  \emph{exponential},
\emph{polynomial} and \emph{constant} was proved by Scheinerman
and Zito in~\cite{Scheinerman:speed-hereditary}. A refined
description of the exponential growth category was also proved in
this paper (we have not included this refinement in our statement
of the classification Theorem~\ref{th:refine} since we found no
shorter proof of this result in the context of minor-closed
classes).  The refined descriptions of the semi-factorial and
polynomial growth categories stated in Theorem~\ref{th:refine}
were established in \cite{Bolagh:speed-hereditary}. Finally, the
\emph{jump} between the semi-factorial growth category and the
almost-factorial growth category was established in
\cite{Bolagh:hereditary-Bell}.

The rest of this section is devoted to the proof of Theorem
\ref{th:refine}. This proof is self-contained and does not use the
results from
\cite{Scheinerman:speed-hereditary,Bolagh:speed-hereditary,Bolagh:hereditary-penultimate,Bolagh:hereditary-Bell}.
We begin by the following easy estimates.

\begin{lem}\label{lem:estimates}
1. The number of path forests of size $n$ satisfies $|\mathcal{P}_n| \geq n!$.\\
2. The number of star forests of size $n$ satisfies  $|\mathcal{S}_n| \geq B(n)$.\\
3. The number of matchings of size $n$ satisfies $|\mathcal{M}_n| \geq n!!=n(n-2)(n-4)\ldots$.\\
4. The number of stars of size $n$ satisfies $|\mathcal{X}_n| \geq 2^{n-1}$.
\end{lem}

We recall that $\log(n!)=n\log(n)+O(n)$, $\log B(n) = n \log(n) - n \log(\log(n)) + O(n)$ and $\log(n!!)=n\log(n)/2+O(n)$.

\begin{proof}
1. The number of path forests of size $n\geq 2$ made of a single path is $n!/2$; the number of path forests of size $n\geq 2$ made of an isolated vertex and a path is $n!/2$.\\
2. A star-forest defines a partition of $[n]:=\{1,2,\dots,n\}$ (together with some marked vertices: the centers of the stars) and the partitions of $[n]$ are counted by the Bell numbers $B(n)$.\\
3. The vertex $n$ of a matching of size $n$ can be isolated or joined to any of the $(n-1)$ other vertices, hence $|\mathcal{M}_n|\geq |\mathcal{M}_{n-1}|+n|\mathcal{M}_{n-2}|$. The property $|\mathcal{M}_n|\geq n!!$ follows by induction.\\
4. The number of stars for which 1 is the center of the star is  $2^{n-1}$.
\end{proof}

\noindent \textbf{Proof of Theorem \ref{th:refine}}  \\
\ite The lower bound for classes of graphs containing all paths
follows from Lemma~\ref{lem:estimates} while the upper bound
follows from Theorem \ref{th:small}.

\vspace{.2cm}

\ite The lower bound for classes of graphs containing all the star
forests but not all the paths follows from
Lemma~\ref{lem:estimates}. The upper bound is given by the
following Claim (and the observation that if a class $\mathcal{G}$
does not contain a given path $P$, then $\mathcal{G} \subseteq
\Ex(P)$).

\begin{claim}\label{claim:path}
For any path $P$, the growth rate of $\Ex(P)$ is bounded by $\epsilon^n n^n$ for all $\epsilon>0$.
\end{claim}

The proof of Claim \ref{claim:path} use the notion of
\emph{depth-first search spanning tree} (or \emph{DFS tree} for
short) of a graph. A DFS tree of a connected graph $G$ is a rooted
spanning tree obtained by a \emph{depth-first search algorithm} on
$G$ (see, for instance, \cite{Cormen:introduction-algo}). If $G$
is not connected, a choice of a DFS tree on each component of $G$
is a \emph{DFS spanning forest}. We recall that if $T$ is a DFS
spanning forest of $G$, every edge of $G$ which is not in $T$
joins a vertex of $T$ to one of its ancestors
(see~\cite{Cormen:introduction-algo}).

\begin{proof}
Let $P$ be the path of size $k$. Let $G$ be a graph in $\Ex(P)$
and let $T$ be a DFS spanning forest  of $G$.
We wish to bound the number of pairs $(G,T)$ of this kind.\\
\ite First, the height of $T$ is at most $k-1$ (otherwise $G$
contains $P$). The number of (rooted labelled) forests of bounded
height is at most $\epsilon^n n^n $ for all $\epsilon>0$; this is
because the associated exponential generating function  is
analytic everywhere and hence has infinite radius of convergence
(see Section III.8.2 in \cite{Flajolet:analytic}).

\ite Second, since $T$ is a  DFS spanning forest, any edge in $G$
which is not in $T$ joins a vertex of $T$ to one of its ancestors.
Since the height of $T$ is at most $k-1$, each vertex has at most
$k$ ancestors, so can be joined to its ancestors in at most $2^k$
different ways. This means that, given $T$, the graph $G$ can be
chosen in at most $2^{kn}$ ways, and so the upper bound
$\epsilon^n n^n$ for all $\epsilon>0$ holds for the number of
pairs $(G,T)$.
\end{proof}

\vspace{.5cm}

\ite We now consider minor-closed classes which do not contain all
the paths nor all the star forests. Given two sequences
$(f_n)_{n\in \N}$ and $(g_n)_{n\in \N}$, we write $f_n \eqexp g_n$
if there exist $a,b>0$ such that $f_n\leq a^ng_n$ and $g_n\leq
b^nf_n$. Observe that if $\mG$ contains all the matchings, then
$g_n\geq n!!\eqexp n^{n/2}$ by Lemma~\ref{lem:estimates}. We prove
the following more precise result.

\begin{claim}\label{claim:pseudo-factorial}
Let $\mG$ be a minor-closed class containing all matchings but not containing all the paths nor all the star forests. Then, there exists an integer $k\geq 2$ such that $g_n\eqexp n^{(1-1/k)n}$.
\end{claim}

\paragraph{Remark.} For any integer $k\geq 2$, there exists a minor-closed class of graphs $\mG$ such that $g_n\eqexp n^{(1-1/k)n}$. For instance, the class $\mG$ in which the connected components have no more than $k$ vertices satisfies this property (see Lemma \ref{lem:unbounded-multiplicity} below).

\begin{proof}
Let $\mG$ be a minor-closed class $\mG$ containing all matchings but not a given path $P$ nor a given star forest $S$. We denote by $p$ and $s$ the size of $P$ and $S$ respectively. Let $\mF$ be set of graphs in $\mathcal{G}$ such that every vertex has degree at most $s$. The following lemma compares the growth rate of $\mF$ and $\mG$.

\begin{lem}
The number $f_n$ of graphs of size $n$ in $\mF$ satisfies $f_n\eqexp g_n$.
\end{lem}

\begin{proof}
Clearly $f_n\leq g_n$ so we only have to prove that there exists
$b>0$ such that $g_n\leq b^nf_n$. Let $c$ be the number of stars
in the star forest $S$ and let $s_1,\ldots,s_c$ be the respective
number of edges of these stars (so that $s=c+ s_1+\ldots+s_c$).

\ite We first prove that \emph{any graph in $\mG$ has less than
$c$ vertices of degree greater than $s$}. We suppose that a graph
$G\in \mG$ has $c$ vertices $v_1,\ldots,v_c$ of degree at least
$s$ and  we want to prove that $G$ contains the forest $S$ as a
subgraph (hence as a minor; which is impossible). For
$i=1\ldots,n$,  let $V_i$ be the set of vertices distinct from
$v_1,\ldots,v_c$ which are adjacent to $v_i$. In order to prove
that $G$ contains the forest $S$ as a subgraph it suffices to show
that there exist disjoint subsets $S_1\subseteq
V_1,\ldots,S_c\subseteq V_c$ of respective size $s_1,\ldots,s_c$.
Suppose, by induction, that for a given $k\leq c$ there exist
disjoint subsets $S_1\subseteq V_1,\ldots,S_{k-1}\subseteq
V_{k-1}$ of respective size $s_1,\ldots,s_{k-1}$. The set
$R_k=V_k-\bigcup_{i\leq k} S_i$  has size at least $s-c-\sum_{i<
k}s_i\geq s_k$, hence there is a subset $S_k\subseteq V_k$
distinct from the $S_i,~i<k$ of size $s_k$. The induction follows.

\ite We now prove that  $g_n \leq {n \choose c} 2^{cn} f_n$. For
any graph in $\mG$ one obtains a graph in $\mF$ by deleting all
the edges incident to the vertices of degree greater than $s$.
Therefore, any graph of $\mG_n$ can be obtained from a graph of
$\mF_n$ by choosing $c$ vertices and adding some edges incident to
these vertices. There are at most ${n \choose c} 2^{cn}f_n$ graphs
obtained in this way.
\end{proof}

It remains to prove that  $f_n\eqexp  n^{(1-1/k)n}$ for some
integer $k\geq 2$. Let $G$ be a graph in $\mF$ and let $T$ be a
tree spanning of one of its connected components. The tree $T$ has
height less than $p$ (otherwise $G$ contains the path $P$ as a
minor) and vertex degree at most $s$. Hence, $T$ has at most
$1+s+\ldots+s^{p-1}\leq s^p$ vertices. Thus the connected
components of the graphs in $\mF$ have at most $s^p$ vertices. For
a connected graph $G$, we denote by $m(G)$ the maximum $r$ such
that $\mF$ contains the graph consisting of $r$ disjoint copies of
$G$. We say that $G$ has \emph{unbounded multiplicity} if $m(G)$
is not bounded. Note that the graph consisting of 1 edge has
unbounded multiplicity since $\mG$ contains all matchings.

\begin{lem}\label{lem:unbounded-multiplicity}
Let $k$ be the size of the largest connected graph in $\mF$ having unbounded multiplicity.
Then, $\displaystyle f_n \eqexp n^{(1-1/k)n}$.
\end{lem}

\begin{proof}
\ite Let $G$ be a connected graph in $\mF$ of size $k$ having unbounded multiplicity. The class
of graphs consisting of disjoint copies of $G$ and isolated vertices (these are included in order
to avoid parity conditions) is contained in $\mF$ and has exponential generating function
 $\exp(z+ z^k/a(G))$, where $a(G)$ is the number of automorphisms of $G$. Hence
$f_n$ is of order at least $n^{(1-1/k)n}$, up to an exponential
factor (see Corollary VIII.2 in \cite{Flajolet:analytic}).

 \ite Let $\mathcal{L}$ be
the class of graphs in which every connected component $C$ appears
at most $m(C)$ times. Then clearly $\mF \subseteq \mathcal{L}$.
The exponential generating function for $\mathcal{L}$ is
$P(z)\exp(Q(z))$, where $P(z)$ collects the connected graphs with
bounded multiplicity, and $Q(z)$ those with unbounded
multiplicity. Since $Q(z)$ has degree $k$, we have an upper bound
of order $n^{(1-1/k)n}$.
\end{proof} 
This finishes the proof of Claim \ref{claim:pseudo-factorial}.
\end{proof}


\ite  We now consider the classes of graphs containing all the
stars but not all the matchings. The lower bound for these classes
follows from Lemma~\ref{lem:estimates} while the upper bound is
given by the following claim.

\begin{claim}
Let $M_k$ be a perfect matching on $2k$ vertices. The growth rate of  $\Ex(M_k)$ is at most exponential.
\end{claim}

\begin{proof}
Let $G$ be a graph of size $n$ in $\Ex(M_k)$ and let $M$ be a
maximal matching of $G$. The matching $M$ has no more than $2k-2$
vertices (otherwise, $M_k<G$). Moreover, the remaining vertices
form an independent set (otherwise, $M$ is not maximal).  Hence
$G$ is a subgraph of the sum $H_n$ of the complete graph
$K_{2k-2}$ and $n-(2k-2)$ independent vertices. There are ${n
\choose 2k-2}$ ways of labeling the graph $H_n$ and $2^{e(H_n)}$
ways of taking a subgraph, where $e(H_n)={2k -2 \choose 2} +
(2k-2)(n-2k+2)$ is the number of edges of $H_n$. Since  ${n
\choose 2k-2}$ is polynomial and $e(H_n)$ is linear, the number of
graphs of size $n$ in $\Ex(M_k)$ is bounded by an
exponential.\end{proof}


\ite  We now consider consider classes of graphs $\mG$ containing
neither all the matchings nor all the stars. If $\mG$ does not
contain all the graphs with a single edge, then either $\mG$
contains all the graphs without edges and $g_n=1$ for $n$ large
enough or $g_n=0$ for $n$ large enough. Observe that if $\mG$
contains the graphs with a single edge, then $g_n\geq
\frac{n(n-1)}{2}$. It only remains to prove the following claim:

\begin{claim}\label{claim:polynomial-growth}
Let $\mG$ be a minor-closed class  containing neither all the
matching nor all the stars. Then, there exists an integer $N$ and
a polynomial $P$  such that $g_n=P(n)$ for all $n\geq N$.
\end{claim}

\paragraph{Remark.} For any integer $k\geq 2$, there exists a minor-closed class of graphs $\mG$ such that $g_n=P(n)$ where $P$ is a polynomial of degree $k$. Indeed, we let the reader check that the class  $\mG$ of graphs made of one star of size at most $k$ plus some isolated vertices satisfies this property.

\begin{proof}
Since $\mG$ 
does not contain all
matchings, one of the  minimal excluded minors of $\mG$ is a
graph $M$ which is made of a set of $k$ independent edges plus $l$
isolated vertices. Moreover, $\mG$ does not contain all the stars,
thus one of the  minimal excluded minors of $\mG$ is a graph $S$
made of one star on $s$ vertices plus $r$ isolated vertices.

\ite We first prove that \emph{for every graph $G$ in $\mG$ having
$n\geq \max(s+r,2k+l)$ vertices, the number of isolated vertices
is at least $n-2ks$.} Observe that for every graph $G$ in $\mG$
having at least $s+r$ vertices, the degree of the vertices is less
than $s$ (otherwise, $G$ contains the star $S$ as a minor).
Suppose now that a graph $G$  in $\mG$ has $n\geq \max(s+r,2k+l)$
vertices from which at least $2ks$ are not isolated. Then, one can
perform a greedy algorithm in order to find $k$ independent edges.
In this case, $G$ contains the graph $M$ as a minor, which is
impossible.

\ite Let $M,S,H_1,\ldots,H_h$ be the minimal excluded minors of
$\mG$ and let $M',S',H_1',\ldots,H_h'$ be the same graphs after
deletion of their isolated vertices. We prove that \emph{there
exists $N\in \N$ such that $\mG_n=\mF_n$ for all $n\geq N$, where
$\mF=\Ex(H_1',\ldots,H_h')$}. Let $m$ be the maximal number of
isolated vertices in the excluded minors $M,S,H_1,\ldots,H_h$ and
let $N=\max(s+r,2k+l,2ks+m)$. If $G$ has at least $N$ vertices,
then $G$ has at least $m$ isolated vertices, hence $G$ is in $\mG$
if and only if it is in $\mF$.

\ite We now prove that there exists a polynomial $P$ with rational
coefficients such that $f_n\equiv |\mF_n|=P(n)$. Let $\mC$ be the
set of graphs in $\mF$ without isolated vertices; by convention we
consider the graph of size 0 as being in $\mC$. The graphs in
$\mC$ have at most $\max(s+r,2k+l,2ks)$ vertices, hence $C$ is a
finite set. We say that a graph in $G$ \emph{follows the pattern}
of a graph $C\in \mC$ if $C$ is the graph obtained from $G$ by
deleting the isolated vertices of $G$ and reassigning the labels
in $\{1,\ldots,r\}$ respecting the order of the labels in $G$. By
the preceding points, any graph in $\mF$ follows the pattern of a
graph in $\mC$ and, conversely, any graph following the pattern of
a graph in $\mC$ is in $\mF$ (since the excluded minors
$M',S',H_1',\ldots,H_h'$ of $\mF$ have no isolated vertices). The
number of graphs of size $n$ following the pattern of a given
graph $C\in \mC$ is ${n \choose |C|}$, where $|C|$ is the number
of vertices of $C$. Thus, $f_n=\sum_{C\in\mC} {n \choose |C|}$
which is a polynomial.
\end{proof}

This conclude the proof of Theorem \ref{th:refine}.

\section{Growth constants}\label{section:growth-constants}

We say that class $\mathcal{G}$ \emph{has growth constant}
$\gamma$ if $\lim_{n \to \infty} \left(g_n/ n!\right)^{1/n} =
\gamma$, and we write $\gamma(\mathcal{G})= \gamma$.

\begin{prop}\label{th:existence}
Let $\mathcal{G}$ be a minor-closed class such that all the excluded minors of $\mathcal{G}$ are 2-connected. Then, $\gamma(\mathcal{G})$ exists. 
\end{prop}

\begin{proof}
In the terminology of \cite{McDiarmid:growth-constant-planar-graphs}, the class $\mathcal{G}$ is small (because of Theorem
\ref{th:small}), and it is addable because of the assumption on the forbidden minors. Hence,
Theorem 3.3 from \cite{McDiarmid:growth-constant-planar-graphs} applies and there exists a growth constant.
\end{proof}

We know state a theorem about the set $\Gamma$ of growth constants of minor-closed classes. In what follows we denote by $\xi \approx 1.76$ the inverse of the unique positive root of $x \exp(x) =1 $.

\begin{thm}\label{thm:growth-constants}
Let $\Gamma$ be the set of real numbers which are growth constants of minor-closed classes of graphs.
\begin{enumerate}
\item The values $0,~1,~\xi$ and $e$ are in $\Gamma$.
\item If $\gamma \in \Gamma$ then $2 \gamma \in \Gamma$.
\item There is no $\gamma \in \Gamma$ with $0 < \gamma <1$.
\item There is no $\gamma \in \Gamma$ with $1 < \gamma <\xi$.
\end{enumerate}
\end{thm}

\paragraph{Remarks.} \ite The property 1 of Theorem \ref{thm:growth-constants} can be extended with the growth constants of the minor-closed classes listed in table by table \ref{table:known-constants}. \\
\ite The properties 2, 3 and 4 of Theorem \ref{thm:growth-constants} remain valid if one replaces $\Gamma$ by the set $\Gamma'=\{\gamma'=\limsup \left(\frac{g_n}{n!}\right)^{1/n} / \mG \textrm{ minor-closed}\}$.

\begin{table}[htb]
\begin{center}
\begin{tabular}{|l|r|l|}
\hline Class of graphs & Growth constant & Reference \\ \hline
$\Ex(P_k)$ & $0$ & This paper\\
Path forests  & $1$ & Standard \\
Caterpillar forests & $\xi \approx 1.76$ & This paper \\
Forests $= \Ex(K_{3})$ & $e \approx 2.71$ &  Standard\\
$\Ex(C_4)$ & $3.63$ & \cite{Gimenez:given-3connected} \\
$\Ex(K_4-e)$ & $4.18$ & \cite{Gimenez:given-3connected} \\
$\Ex(C_5)$ & $4.60$ & \cite{Gimenez:given-3connected} \\
Outerplanar $=\Ex(K_4,K_{2,3})$ & 7.320 & \cite{Bodirsky:series-parallel+outerplanar} \\
$\Ex(K_{2,3})$ & 7.327 & \cite{Bodirsky:series-parallel+outerplanar} \\
Series parallel $=\Ex(K_4)$ &  9.07 & \cite{Bodirsky:series-parallel+outerplanar} \\
$\Ex(W_4)$ & $11.54$ & \cite{Gimenez:given-3connected} \\
$\Ex(K_5-e)$ & $12.96$ & \cite{Gimenez:given-3connected} \\
$\Ex(K_2 \times K_3)$ & $14.13$ & \cite{Gimenez:given-3connected}\\
Planar &   27.226 &  \cite{Gimenez:planar-graphs} \\
Embeddable in a fixed surface  & 27.226 &  \cite{McDiarmid:graphs-on-surfaces} \\
$\Ex(K_{3,3})$ & 27.229 &   \cite{Gerke:K33-free} \\
 \hline
\end{tabular}
\caption{A table of some known growth constants.}\label{table:known-constants}
\end{center}
\end{table}

Before the proof of Theorem \ref{thm:growth-constants}, we make
the following remark. Let $\mathcal{G}$ be a minor-closed class,
let $\mathcal{C}$ be the family of all connected members of
$\mathcal{G}$, and let $G(z)$ and $C(z)$ be the corresponding
generating functions. Then if $\mathcal{C}$ has growth constant
$\gamma$, so does $\mathcal{G}$. This is because the generating
functions $G(z)$ is bounded by $\exp(C(z))$ (they are equal if the
forbidden minors for $\mathcal{G}$ are all connected), and both
functions have the same dominant singularity.

\begin{proof}
1) \ite All classes whose growth is not at least factorial have
growth constant $0$. In particular, $\gamma(\Ex(P)) = 0$ for any
path $P$.

\ite The number of labelled paths is $n!/2$. Hence, by the remark
made before the proof, the growth constant of the class of path
forests is 1.

\ite A \emph{caterpillar} is a tree consisting of a path and
vertices directly adjacent to (i.e. one edge away from) that path.
Let $\mC$ be the class of graphs whose connected components are
caterpillars, which is clearly minor-closed. A rooted caterpillar
can be considered as an ordered sequence of stars. Hence the
associated generating function is $1/(1 - z e^z)$. The dominant
singularity is the smallest positive root of $1-ze^z=0$, and
$\gamma(\mC)$ is the inverse $\xi$ of this value.

\ite The growth constant of the class of acyclic graphs (forests)
is the number $e$. This is because the number of labelled trees is
$n^{n-2}$ which, up to a sub-exponential factor, is asymptotic to
$\sim e^n n!$.

2) This property follows from an idea by Colin McDiarmid. Suppose $\gamma(\mathcal{G})= \gamma$, and
let $\mathcal{AG}$ be family of graphs $G$ having a
vertex $v$ such that $G-v$ is in $\mathcal{G}$; in this case we say that $v$ is an apex of $G$. It
is easy to check that  if $\mathcal{G}$ is minor-closed, so is $\mathcal{AG}$. Now we have
$$ 2^n |\mathcal{G}_n| \le   |\mathcal{AG}_{n+1}| \le (n+1) 2^n |\mathcal{G}_n|. $$
The lower bound is obtained by taking a graph $G \in \mathcal{G}$
with vertices $[n]$, adding $n+1$ as a new vertex, and making
$n+1$ adjacent to any subset of $[n]$. The upper bound follows the
same argument by considering which of the vertices $1,2,\dots,n+1$
acts as an apex. Dividing by $n!$ and taking $n$-th roots, we see
that $\gamma(\mathcal{AG})= 2\gamma(\mathcal{G})$.

3) This has been already shown during the proof of Theorem \ref{th:refine}. Indeed, if a
 minor-closed class $\mathcal{G}$ contains all paths, then $|\mathcal{G}_n|\ge n!/2$ and the growth constant is at least $1$.
 Otherwise $g_n < \epsilon^n n^n $ for all $\epsilon>0$ and  $\gamma(\mathcal{G})=0$.\\

4) We consider the graphs $\Cat_l$ and $\Ap_l$ represented in
Figure \ref{fig:two-obstructions}.

\begin{figure}[ht!]\begin{center} \input{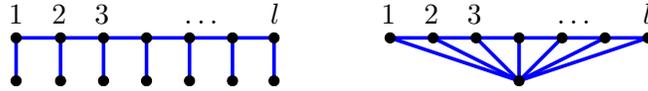}\caption{The graph $\Cat_l$ (left) and the graph
$\Ap_l$ (right).}\label{fig:two-obstructions} \end{center}\end{figure}

If a minor-closed class $\mG$ contains the graphs $\Cat_l$ for all
$l$, then $\mG$ contains all the caterpillars hence
$\gamma(\mG)\geq \xi\approx 1.76$. If $\mG$ contains the graphs
$\Ap_l$ for all $l$, then $\mG$ contains the apex class of path
forests and $\gamma(\mG)\geq 2$. Now, if $\mG$ contains neither
$\Cat_k$ nor  $\Ap_l$ for some $k,l$, then $\mG\subseteq
\Ex(\Cat_l,\Ap_l)$. Therefore, it is sufficient to prove the
following claim.

\begin{claim} \label{claim:gap}
The growth constant of the class \emph{$\Ex(\Cat_k,\Ap_l)$} is 1 for all $k> 2 ,l>1$.
\end{claim}

\paragraph{Remark.}
Claim \ref{claim:gap} gives in fact a characterization of the
minor-closed classes with growth constant 1. These are the classes
containing all the paths but neither all the caterpillars nor all
the graphs in the apex class of the path forests. For instance,
the class of trees not containing a given caterpillar (as a minor)
and the class of graphs not containing a given star (as a minor)
both have growth constant 1.

\begin{proof}
Observe that the class $\Ex(\Cat_k,\Ap_l)$ contains all paths as
soon as $k>2$ and $l>1$. Hence, $\gamma(\Ex(\Cat_k,\Ap_l))\geq 1$
(by Lemma \ref{lem:estimates}) and we only need to prove that
$\gamma(\Ex(\Cat_k,\Ap_l))\leq~1$. We first prove a result about
the simple paths of the graphs in $\Ex(\Cat_k,\Ap_l)$.

\begin{lem}\label{lem:degree2-on-path}
Let $G$ be a graph in \emph{$\Ex(\Cat_k,\Ap_l)$} and let $P$ be a
simple path in $G$. Then, there are less than $kl+4k^3l$ vertices
in $P$ of degree greater than 2.
\end{lem}

\begin{proof}
\ite We first prove that \emph{any vertex not in $P$ is adjacent
to less than $l$ vertices of $P$} and \emph{any vertex in $P$ is
adjacent to less than $2l$ vertices of $P$}. Clearly, if $G$
contains a vertex $v$ not in $P$ and adjacent to $l$ vertices $P$,
then $G$ contains $\Ap_l$ as a minor. Suppose now that there is a
vertex $v$ in $P$ adjacent to $2l$ other vertices of $P$. In this
case, $v$ is adjacent to at least $l$ vertices  in one of the
simple paths $P_1,~P_2$ obtained by removing the vertex $v$ from
the path~$P$. Hence $G$ contains $\Ap_l$ as a minor.

\ite We now prove that \emph{there are less than $kl$ vertices
in~$P$ adjacent to at least one  vertex not in~$P$}. We suppose
the contrary and we prove that there exist $k$ independent edges
$e_i=(u_i,v_i),~i=1\ldots k$ such that  $u_i$ is in $P$ and $v_i$
is not in $P$ (thereby implying that $\Cat_k$ is a minor of $G$).
Let  $r< k$ and let $e_i=(u_i,v_i),~i\leq r$ be independent edges
with $u_i\in P$ and $v_i\notin P$. The set of vertices in $P$
adjacent to some vertices not in $P$ but to none of the vertices
$v_i,i\leq r$ has size at least $kl-rl>0$ (this is because each of
the vertex $v_i$ is adjacent to less than $l$ vertices of $P$).
Thus, there exists an edge $e_{r+1}=(u_{r+1},v_{r+1})$ independent
of the edges  $e_i,i\leq r$ with $u_{r+1}\in P$ and $v_{r+1}\notin
P$. Thus, any set of $r<k$ independent edges with one endpoint in
$P$ and one endpoint not in $P$ can be increased.

\ite We now prove that  \emph{there are no more than $4k^3l$
vertices in $P$ adjacent to another vertex in~$P$ beside its 2
neighbors in $P$}. We suppose the contrary and we prove that
either $\Cat_k$ or $\Ap_l$ is a minor of $G$. Let $E_P$ be the set
of edges not in the path $P$ but joining 2 vertices of~$P$. We say
that two independent edges $e=(u,v)$ and $e'=(u',v')$ of $E_P$
\emph{cross} if the vertices $u,u',v,v'$ appear in this order
along the path $P$; this situation is represented in Figure
\ref{fig:crossing-edges}~(a).

\iten We first show that \emph{there is a subset $E_P'\subseteq
E_P$ of $k^3$ independent edges}. Let $S$ be any set of $r<k^3$
edges in $E_P$. The number of edges in $E_P$ sharing a vertex with
one of the edges in $S$ is at most $2r\times 2l<4k^3l$  (this is
because any vertex in $P$ is adjacent to less than $2l$ vertices
in $P$). Since $|E_P|\geq 4k^3l$, any set of independent edges in
$E_P$ of size less than $k^3$ can be increased.

\iten We now show that \emph{for any edge $e$ in $E_P'$ there are
at most $k$ edges of $E_P'$ crossing $e$}. Suppose that there is a
set $S\subseteq E_P'$ of $k$ edges crossing $e$. Let $P'$ be the
path obtained from $P\cup e$ by deleting the edges of $P$ that are
between the endpoints of $e$. The graph made of  $P'$ and the set
of edges $S$ contains the graph $\Cat_l$ as a minor which is
impossible.

\iten We now show that \emph{there exists a subset $E_P''\subseteq
E_P'$ of $k^2$ non-crossing edges}. Let $S$ be any set of $r<k^2$
edges in $E_P'$. By the preceding point, the number of edges in
$E_P'$ crossing one of the edges in $S$ is less than $rk<k^3$.
Since   $|E_P'|\geq k^3$, any set of non-crossing edges in~$E_P'$
of size less than $k^2$ can be increased.

\iten Lastly, we show that \emph{the graph $\Cat_k$ is a minor of
$G$}. We say that an edge $e=(u,v)$ of~$E_P''$ is \emph{inside}
another edge $e'=(u',v')$ if $u',u,v,v'$ appear in this order
along the path $P$; this situation is represented in Figure
\ref{fig:crossing-edges}~(b). We define the \emph{height} of the
edges in $E_P''$ as follows: the height of an edge $e$ is 1 plus
the maximum height of edges of $E_P''$ which are inside $e$ (the
height is 1 if there is no edge inside $e$). The height of edges
have been indicated in Figure~\ref{fig:crossing-edges}~(c).
Suppose that there is an edge of height $k$ in $E_P''$. Then there
is a set $S$ of $k$ edges  $e_1=(u_1,v_1),\ldots,e_k=(u_k,v_k)$
such that the vertices $u_1,u_2,\ldots,u_k,v_k,v_{k-1},\ldots,v_1$
appear in this order along $P$. In this case, the subgraph made of
$S$ and the subpath of $P$ between $u_1$ and $u_k$ contains
$\Cat_k$ as a minor. Suppose now that there is no edge of height
$k$. Since there are $k^2$ edges in $E_P''$, there is a integer
$i<k$ such that the number of edges of height $i$ is greater than
$k$.  Thus, there is a set $S$ of $k$ edges
$e_1=(u_1,v_1),\ldots,e_k=(u_k,v_k)$ such that the vertices
$u_1,v_1,u_2,v_2,\ldots,u_k,v_k$ appear in this order along $P$.
In this case, the subgraph obtained from $P\cup
\{e_1,\ldots,e_k\}$ by deleting an edge of $P$ between $u_i$ and
$v_i$ for all $i$ contains $\Cat_k$ as a minor.
\end{proof}

\begin{figure}[ht!]\begin{center} \input{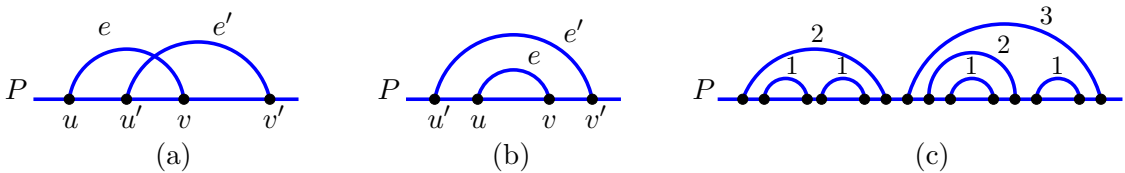}\caption{(a) Two crossing edges. (b) An edge inside
another.  (c) A set of non-crossing edges.}\label{fig:crossing-edges} \end{center}\end{figure}  \vspace{-.3cm}

For any integer $N$, we denote by $\mG^{N}_T$ the set of pairs
$(G,T)$ where $G$ is a graph and $T$ is a DFS spanning forest on
$G$ having height at most $N$ (the definition of \emph{DFS
spanning forest} was given just after Claim \ref{claim:path}).

\begin{lem}\label{lem:subdivide-tree}
For any graph $G$ in $\Ex(\Cat_k,\Ap_l)$, there exists a pair
$(G',T')$ in $\mG^{kl+4k^3l}_T$ such that $G$ is obtained from
$G'$ by subdividing some edges of $T$.
\end{lem}

\begin{proof}
Let $G$ be a graph in $\Ex(\Cat_k,\Ap_l)$, let $T$ be a DFS
spanning forest of $G$ and let $R$ be the set of roots of $T$ (one
root for each connected components of~$G$). One \emph{contracts} a
vertex $v$ of degree $2$ by deleting $v$ and joining its two
neighbors by an edge. Let $G'$ and $T'$ be the graphs and trees
obtained from $G$ and $T$ by contracting the vertices $v\notin R$
of degree 2 which are incident to 2 edges of $T$.  We want to
prove that $(G',T')$ is in $\mG^{kl+4k^3l}_T$.

\ite Since $T$ is a DFS spanning forest of $G$, every edge of $G$
which is not in $T$ connects a vertex to one of its ancestors
\cite{Cormen:introduction-algo}. This property characterize the
DFS spanning forests and is preserved by the contraction of the
vertices of degree 2. Hence, $T'$ is a DFS spanning forest of
$G'$.

\ite By Lemma \ref{lem:degree2-on-path}, the number of vertices
which are not of degree $2$ along a path of $T$ from a root to a
leaf is less than $kl+4k^3l$. Thus, the height of $T'$ is at most
$kl+4k^3l$.
\end{proof} 

We have already shown in the proof of Claim \ref{claim:path} that
the radius of convergence of the generating function $G^{N}_T(z)$
of the set $\mG^{N}_T$ is infinite. Moreover, the generating
function of the set of graphs that can be obtained from pairs
$(G',T')$ in $\mG^{N}_T$  by subdividing the tree $T'$ is bounded
(coefficient by coefficient) by $G^{N}_T(\frac{z}{1-z})$ (since a
forest $T'$ on a graph $G'$ of size $n$ has at most $n-1$ edges to
be subdivided). Thus, Lemma \ref{lem:subdivide-tree} implies that
the generating function of $\Ex(\Cat_k,\Ap_l)$ is bounded by
$G^{kl+4k^3l}_T(\frac{z}{1-z})$ which has radius of convergence 1.
Hence, the growth constant $\gamma(\Ex(\Cat_k,\Ap_l))$ is at most
1.
\end{proof}
This concludes the proof of Claim \ref{claim:gap} and Theorem
\ref{thm:growth-constants}.
\end{proof}

We now investigate the topological properties of the set $\Gamma$
and in particular its limit points. First note that $\Gamma$ is
countable
(as a consequence of the Minor Theorem of Robertson and Seymour \cite{Seymour:Graph-minors}). 

\begin{lem}\label{th:sep}
Let $H_1,H_2,\dots H_k$ be a family of 2-connected graphs, and let
$\mathcal{H} = {\rm Ex}(H_1,H_2,\dots H_k)$. If $G$ is a
2-connected graph in $\mathcal{H}$, then
 $\gamma(\mathcal{H}\cap \Ex(G)) < \gamma(\mathcal{H})$.

\end{lem}

\begin{proof}
The condition on 2-connectivity guarantees that the growth
constants exist. By Theorem 4.1 from
\cite{McDiarmid:growth-constant-planar-graphs}, the probability
that a random graph in $\mathcal{H}_n$ contains $G$ as a subgraph
is a least $1 - e^{-\alpha n}$ for some $\alpha>0$. Hence the
probability that a random graph in $\mathcal{H}_n$ does not
contain $G$ as a minor is at most $e^{-\alpha n}$. If we denote
$\mathcal{G} = \mathcal{H}\cap \Ex(G)$, then we have
$$
{|\mathcal{G}_n| \over |\mathcal{H}_n|} = {|\mathcal{G}_n| \over n!} {n! \over |\mathcal{H}_n|}
\le e^{-\alpha n}.
$$
Taking limits, this implies
$$
{\gamma(\mathcal{G}) \over \gamma(\mathcal{H})} \le \lim \left(e^{-\alpha n}\right)^{1/n} =
e^{-\alpha} < 1.
$$
\end{proof}

We recall that given a set $A$ of real numbers, $a$ is a
\emph{limit point} of $A$ if for every $\epsilon>0$ there exists
$x\in A-\{a\}$ such that $|a-x| < \epsilon$.

\begin{thm}\label{th:limit}
Let $H_1,\ldots,H_k$ be 2-connected graphs which are not cycles. Then, $\gamma=\gamma(\Ex(H_1,\ldots,H_k))$ is a limit point of $\Gamma$.
\end{thm}

\begin{proof}
For $k \ge3$, let $\mathcal{G}_k = \mathcal{G} \cap \Ex(C_k)$,
where $C_k$ is the cycle of size $k$. Because of Proposition
\ref{th:existence}, the class $\mathcal{G}_k$ has a growth
constant $\gamma_k$, and because of Lemma \ref{th:sep} the
$\gamma_k$ are strictly increasing and $\gamma_k < \gamma$ for all
$k$. It follows that $\gamma' = \lim_{k \to \infty} \gamma_k$
exists and $\gamma' \le \gamma$. In order to show equality we
proceed as follows.

Let $g_n = |\mathcal{G}_n|$ and let $g_{k,n} = |(\mathcal{G}_k)_n|$. Since $\gamma =
\lim_{n\to\infty}(g_n/n!)^{1/n}$, for all $\epsilon>0$ there exists $N$ such that for $n>N$ we
have
$$
\left(g_n/ n! \right)^{1/n} \geq \gamma -\epsilon.
$$
Now define $\displaystyle f_n=\frac{g_n}{e^2 n!}$ and
$\displaystyle f_{k,n}=\frac{g_{k,n}}{e^2n!}$. From \cite[Theorem
3]{McDiarmid:growth-constant-planar-graphs}, the sequence $f_n$ is
supermultiplicative and $\displaystyle \gamma=\lim_{n\to
\infty}\left(f_n\right)^{1/n}=\lim_{n\to
\infty}\left(g_n/n!\right)^{1/n}$ exists and equals $\sup_n
\left(f_n\right)^{1/n}$. Similarly, $\gamma_k=\lim_{n\to
\infty}\left(f_{k,n}\right)^{1/n}=\sup_n
\left(f_{k,n}\right)^{1/n}$.

But since a graph on less than $k$ vertices cannot contain $C_k$
as a minor, we have $g_{k,n} = g_n$ for  $k>n$. Equivalently,
$f_{k,n} = f_n$ for  $k>n$. Combining all this, we have
$$\gamma_k \geq \left(f_{k,n}\right)^{1/n} \geq \left(f_n\right)^{1/n}\geq \gamma -\epsilon$$
for $k>N$. This implies $\gamma' = \lim \gamma_k \ge \gamma$.
\end{proof}

Notice that Theorem \ref{th:limit}  applies to all the classes in
Table \ref{table:known-constants} starting at the class of
outerplanar graphs. However, it does not apply to the classes of
of forests. In this case we offer an independent proof based on
generating functions.

\begin{lem}
The number $e$ is a limit point of $\Gamma$.
\end{lem}

\begin{proof}
Let $\mathcal{F}_k$ be the class of forests whose trees are made
of a path and rooted trees of height at most $k$ attached to
vertices of the path. Observe that the classes $\mathcal{F}_k$ are
minor-closed, that $\mathcal{F}_k\subset \mathcal{F}_{k+1}$, and
that $\cup_{k} \mathcal{F}_k=\mathcal{F}$, where $\mF$ is the
class of forests. We prove that $\gamma(\mF_k)$ is a strictly
increasing sequence tending to $e=\gamma(\mF)$.

Recall that the class $\mathcal{F}_k$ and the class
$\mathcal{T}_k$ of its connected members have the same growth
constant. Moreover, the class $\vec{\mathcal{T}}_k$ of trees with
a distinguished oriented path to which rooted trees of height at
most $k$ are attached has the same growth constant as
$\mathcal{T}_k$ (this is because there are only  $n(n-1)$ of
distinguishing and orienting a path in a tree of size $n$). The
generating function associated to $\vec{\mathcal{T}}_k$ is
$1/(1-F_k(z))$, where $F_k(z)$ of is the generating function of
rooted trees of height at most $k$. Hence,
$\gamma(\mathcal{F}_k)=\gamma(\vec{\mathcal{T}}_k)$ is the inverse
of the unique positive root $\rho_k$ of $F_k(\rho_k) = 1$.

Recall that the generating functions $F_k$ are obtained as
follows; see Section III.8.2 in \cite{Flajolet:analytic}).

$$
F_0(z) = z; \qquad F_{k+1}(z) = ze^{F_k(z)} \quad \hbox{for $k>0$}.
$$
It is easy to check that the roots $\rho_k$ of $F_k(\rho_k) = 1$
are strictly decreasing.
Recall that the generating function $F(z)$ of rooted trees has a
singularity at $1/e$ and that $F(1/e)=1$
(see~\cite{Flajolet:analytic}). Moreover, for all $n$, $0\leq
[z^n]F_k(z)  \leq [z^n]F(z)$ and $\lim_{k\to \infty}
[z^n]F_k(z)=[z^n]F(z)$, thus $\lim_{k\to \infty} F_k(1/e) = F(1/e)
=1$. Furthermore, the functions $F_k(z)$ are convex and
$F_k'(1/e)\geq 1$ (since the coefficients of $F_k$ are positive
and $[z^1]F_k(z)=1$). Thus, $F_k(z)>F_k(1/e)+(z-1/e)$ which
implies   $1/e\leq \rho_k \leq 1/e+(F_k(1/e)-F(1/e))$. Thus, the
sequence $\rho_k$ tends to $1/e$ and the growth constants
$\gamma(\mathcal{F}_k)=1/\rho_k$ tend to $e$.
\end{proof}

\paragraph{Remark.} The number $\nu \approx 2.24$, which is the inverse of the smallest positive
root of $z\exp(z/(1-z))=1$, can be shown to be a limit point of $\Gamma$ by similar methods. It is
the smallest number which we know to be a limit point of $\Gamma$. It is the growth constant of
the family whose connected components are made of a path $P$ and any number of paths of any
length attached to the vertices of $P$.

\paragraph{Remark.} All our examples of limit points in $\Gamma$ come from
strictly increasing sequences of growth constants that converge to
another growth constant. Is it possible to have an infinite
strictly decreasing sequences of growth constants? As we see now,
this is related to a classical problem. A quasi-ordering is a
reflexive and transitive relation. A quasi-ordering $\le$ in $X$
is a \emph{well-quasi ordering} if for every infinite sequence
$x_1,x_2,\dots$ in $X$ there exist $i<j$ such that $x_i \le x_j$.
Now consider the set $X$ of minor-closed classes of graphs ordered
by inclusion. It is an open problem whether this is a well-quasi
ordering \cite{Diestel:wqo-minors}. Assuming this is the case, it
is clear that an infinite decreasing sequence $\gamma_1
> \gamma_2 > \cdots$ of growth constants cannot exist. For consider the corresponding sequence of
graph classes $\mathcal{G}_1 , \mathcal{G}_2,\dots$. For some $i<j$ we must have $\mathcal{G}_i
\subseteq \mathcal{G}_j$, but this implies $\gamma_i \le \gamma_j$.

\section{Conclusion: some open problems}\label{section:conclusion}
We close by  listing some  of  the open  questions which  have arisen  in  this  work.\\

1)  We know that  a class $\mathcal{G}$ has a growth constant
provided that all its excluded minor are 2-connected.
The  condition that  the excluded-minors are 2-connected is  certainly not necessary as is  seen  by noting that the  apex family  of any class which has a growth constant  also  has  a growth constant. It is also easy to see that such an apex family  is also  minor-closed and that  at  least one  of  its excluded minors is disconnected.

Thus  our  first  conjecture is  that every minor-closed family
has a  growth constant, that is, $\lim
\left(\frac{g_n}{n!}\right)^{1/n}$ exists for every minor-closed
class $\mG$.


2)  A  minor-closed class is  \emph{smooth} if $\lim
\frac{g_n}{ng_{n-1}}$ exists. It follows  that this limit must be
the  growth constant and  that a random member of $\mathcal{G}$
will  have  expected number of  isolated vertices converging to
$1/\gamma$. Our  second  conjecture is that  if every  excluded
minor of  a  minor-closed class is 2-connected then the  class is
smooth.

If  true, then it  would  follow  that  a  random  member of  the
class would  qualitatively   exhibit all the  Poisson type
behaviour exhibited by the  random  planar  graph. However proving
smoothness for  a  class seems  to  be  very  difficult and the
only  cases which  we  know  to  be  smooth are when  the
exponential generating  function has been  determined exactly.

3) We  have  shown that the intervals $(0,1)$ and $(1,\xi)$ are
"gaps" which contain no growth constant. We know of no other gap,
though if there is no infinite decreasing sequence of growth
constants they  must  exist. One  particular question which we
have  been unable to settle is whether $(\xi,2)$ is also a gap.

4)  We   have   shown that for each nonnegative  integer  $k$,
$2^k$  is a growth  constant.  A natural question is whether any
other integer is a growth constant. More generally, is there any
algebraic number in $\Gamma$ besides the powers of 2?

5) All our results concern labelled graphs. In unlabelled setting,
the  most important  question to settle  is whether there is an
analogue of the theorem  of  Norine \emph{et al}. More precisely,
suppose $\mathcal{G}$ is a minor-closed class of  graphs and that
$u_n$ denotes the number of unlabelled members of $\mathcal{G}_n$.
Does there exist  a  finite $d$ such that    $u_n$ is bounded
above by $d^n$?

\paragraph{Aknowledgements.}
We are very grateful to Colin McDiarmid who suggested the
apex-construction, to Angelika Steger for useful discussions,  and
to Norbert Sauer and Paul Seymour for information on well quasi
orders.

\bibliography{biblio-growth-minor.bib}
\bibliographystyle{plain}

\end{document}